\documentclass[a4paper,10pt]{article}

\usepackage{amsthm}
\usepackage{thmtools}
\usepackage{hyperref}
\usepackage{mathrsfs}
\usepackage{changepage}

\declaretheoremstyle[bodyfont=\sl]{slanted}

\swapnumbers
\declaretheorem[name=Definition,style=definition,qed=$\dashv$,
numberwithin=section]{dfn}
\declaretheorem[name=Definition,style=definition,numbered=no,qed=$\dashv$]{dfn*}
\declaretheorem[name=Definition,style=definition,numbered=no]{dfnnoqed*}
\declaretheorem[name=Theorem,style=slanted,sibling=dfn]{tm}
\declaretheorem[name=Theorem,style=slanted,numbered=no]{tm*}

\declaretheorem[name=Corollary,style=slanted,numbered=no]{cor*}
\declaretheorem[name=Remark,style=definition,sibling=dfn]{rem}
\declaretheorem[name=Question,style=definition,sibling=dfn]{ques}

\swapnumbers
\declaretheoremstyle[headfont=\scshape]{claimstyle}

\declaretheorem[name=Claim,style=claimstyle]{clm}

\usepackage{amsmath}
\usepackage{amsxtra}
\usepackage{amssymb}
\usepackage{turnstile}
\usepackage{enumitem}
\usepackage{accents}
\usepackage{xcolor}
\usepackage[retainorgcmds]{IEEEtrantools}

\newcommand{\RR}{\mathbb R}
\newcommand{\sub}{\subseteq}
\newcommand{\cross}{\times}

\newcommand{\om}{\omega}
\newcommand{\pow}{\mathcal{P}}
\newcommand{\OR}{\mathrm{OR}}
\newcommand{\Card}{\mathrm{Card}}

\newcommand{\Tt}{\mathcal{T}}
\newcommand{\Uu}{\mathcal{U}}
\newcommand{\vareps}{\varepsilon}
\newcommand{\rg}{\mathrm{rg}}

\newcommand{\ins}{\trianglelefteq}

\newcommand{\pins}{\triangleleft}

\newcommand{\crit}{\mathrm{cr}}
\newcommand{\rest}{\!\upharpoonright\!}

\newcommand{\lh}{\mathrm{lh}}

\newcommand{\sats}{\models}

\newcommand{\J}{\SS}
\newcommand{\HC}{\mathrm{HC}}
\newcommand{\ZFC}{\mathsf{ZFC}}
\newcommand{\ZF}{\mathsf{ZF}}
\newcommand{\es}{\mathbb{E}}

\newcommand{\her}{\mathcal{H}}

\DeclareMathOperator{\card}{card}
\DeclareMathOperator{\cof}{cof}

\title{Power $\Sigma_1$ in Card with two Woodin cardinals}
\author{Farmer Schlutzenberg\\
TU Vienna}
\begin{document}

\maketitle

\begin{abstract}
V\"a\"an\"anen and Welch asked in the paper \cite{power_set_Sigma_1_Card} \emph{When cardinals determine the power set: inner models and H\"artig quantifier logic}
which large cardinals are consistent with the power set operation $x\mapsto\pow(x)$ being $\Sigma_1$-definable in the predicate $\Card$ of all cardinals. We show that, relative to large cardinals,  this property is consistent together with the existence of two Woodin cardinals.
\end{abstract}

In \cite{power_set_Sigma_1_Card}, Jouko V\"a\"an\"anen
and Philip Welch examine the relationship between large cardinals and the
 $\Sigma_1^{\mathrm{Card}}$-definability of the power set operation, especially below one Woodin cardinal.
Here one says that \emph{power is
$\Sigma_1^{\mathrm{Card}}$}
 iff there is a $\Sigma_1$ formula $\varphi$ of the language of set theory\footnote{Recall here that the usual language for definability over mice  includes predicate symbols for the extender sequence, but here when we talk about $\Sigma_1^{\mathrm{Card}}$, we certainly aren't allowed such predicates in our language.} such that for all $x,y$, we have
\[ y=\pow(x)=\{A\bigm|A\sub x\} \]
iff
\[ \exists\alpha\in\OR\ \big[\varphi(x,y,\mathrm{Card}\cap\alpha)\big].\]
They show (see \cite[Lemma 3.7]{power_set_Sigma_1_Card}) that if $M$ is a proper class mouse satisfying ``there is no proper class transitive inner model with a Woodin cardinal'' then $M\models$ ``Power is $\Sigma_1^{\Card}$''.
They go on to ask in \cite[Question 4.5]{power_set_Sigma_1_Card} about which large cardinals are consistent with power being $\Sigma_1^{\mathrm{Card}}$.
As a consequence of \cite[Lemma 3.7]{power_set_Sigma_1_Card} and some standard arguments,
if $M_1^\#$ exists and is fully iterable
then $M_1\sats$ ``power is $\Sigma_1^{\mathrm{Card}}$''. (Recall here that $M_1$ is the canonical minimal proper class mouse with a Woodin cardinal. We will sketch the proof of this conclusion in Remark \ref{rem:M_1} below.)
We make here a further step  regarding this question:

\begin{tm}\label{tm:y=P(x)_Sigma_1^Card}Suppose that $M_2^\#$ exists and is $(0,\om_1+1)$-iterable.
Then there is a proper class transitive inner model which models ZFC + ``There are 2 Woodin cardinals'' + ``power is $\Sigma_1^{\mathrm{Card}}$''.
\end{tm}

\begin{rem}\label{rem:M_1}Note that the conclusion of the theorem is not simply that $M_2\sats$ ``power is $\Sigma_1^{\mathrm{Card}}$''; our proof will not give this, and the author does not know whether it holds in $M_2$.\end{rem}
\begin{rem}
Let us deduce from \cite[Lemma 3.7]{power_set_Sigma_1_Card}
the fact that $M_1\models$ ``Power is $\Sigma_1^{\mathrm{Card}}$.
Write $\delta_0^{M_1}$ for the unique Woodin of $M_1$, let $\delta<\delta_0^{M_1}$. Then $M_1\sats$ ``There is no proper class inner model in which $\delta$ is Woodin''. For suppose $N\sub M_1$ is such an inner model. Then there is a fully backgrounded $L[\es]$-construction of $N$ which produces a proper class $1$-small premouse $P$ and $\eta\leq\delta$ such that $P\sats$ ``$\eta$ is Woodin'', and
such that $P$ is $0$-maximally short-tree iterable\footnote{For \emph{$0$-maximal} see \cite{outline}. A $0$-maximal tree $\Tt$ of limit length is \emph{short} if there is $\alpha\in\OR$
such that $\J_\alpha(M(\Tt))$ is either a Q-structure for $\delta(\Tt)$ or projects $<\delta(\Tt)$. See \cite{outline}.}
(in $N$, and hence also in $M_1$);
see \cite[Theorem 11.3]{fsit} and \cite{itertrees}.
Working in $M_1$, we can compare $P$ with $M_1$, and because $\eta<\delta_0^{M_1}$
and both are proper class, this comparison must have length $\delta_0^{M_1}$,
producing a maximal (that is, limit length but non-short) tree on $P$ at that stage.
But then fixing some $\gamma\in(\eta,\delta_0^{M_1})$ such that $\gamma$ is inaccessible in $M_1$,
the comparison of $M_1|\gamma$ with $L_\gamma[P|\delta]$ corresponds to an initial segment of the overall comparison, and $L_\gamma[P|\delta]$ out-iterates $M_1|\gamma$, which leads to a contradiction, via standard arguments with universal weasels.

So $M_1|\delta_0^{M_1}\models\ZFC$ +  ``There is no proper class inner model with a Woodin cardinal'', and so by \cite[Lemma 3.7]{power_set_Sigma_1_Card},  $M_1|\delta_0^{M_1}\sats$ ``Power is $\Sigma_1^{\mathrm{Card}}$''.

Now we claim that for sets $x,y\in M_1$, we have that
$y=\pow(x)\cap M_1$ iff
$M_1\sats$ ``there is are cardinals $\delta<\varepsilon$
and a $1$-small premouse $N$
such that:
\begin{enumerate}[label=--]
\item $\OR^N=\delta$
\item the universe of $N$ is just $\her_\delta$,
\item $N\sats$ ``I am $(\om,\OR)$-iterable'',
\item $L_\varepsilon[N]$ is a premouse modelling ``$\delta$ is Woodin'', and
\item  $x\in L_\varepsilon[N]$
and $y=\pow(x)\cap L_\varepsilon[N]$''.
\end{enumerate}
Note that since $M_1|\delta_0^{M_1}\models$ ``Power is $\Sigma_1^{\mathrm{Card}}$'',
this yields a $(\Sigma_1^{\mathrm{Card}})^{M_1}$ definition
of $x\mapsto\pow(x)^{M_1}$, as desired.

So it suffices to see the claim.
For this, it suffices to see that $M_1|\delta_0^{M_1}$ is the unique  premouse $N$
with the properties above.
(For then
 since $\varepsilon$
is a cardinal in $M_1$, we have $\pow(x)\cap M_1\sub L_\varepsilon[N]$ for all $x\in L_\varepsilon[N]$.)
This follows from results in \cite{extmax}.
That is, suppose $N$ is as above, but $M_1|\delta_0^{M_1}\neq N$. We can compare $M_1|\delta_0^{M_1}$ with $N$, through $\delta_0^{M_1}$-many stages of comparison, or until the models are lined up through height $\delta_0^{M_1}$, whichever occurs first.
Let $\Tt$ be the resulting tree on $M_1$,
and $\Uu$ that on $N$.
Supposing that $N\neq M_1|\delta_0^{M_1}$,
let $\theta<\delta_0^{M_1}$ be such that $N|\theta\neq M_1|\theta$, and let $\gamma\in(\theta,\delta_0^{M_1})$ be inaccessible in $M_1$.
Let $\zeta$ be least such that either $M^\Tt_\zeta|\gamma=M^\Uu_\zeta|\gamma$ or $\gamma<\lh(E^\Tt_\zeta)$ or $\gamma<\lh(E^\Uu_\zeta)$, whichever is defined at stage $\zeta$.
Then since $\her_\delta^{M_1}$ is the universe of $N$, standard considerations with universal weasels
give that $[0,\zeta]^\Tt$ is non-dropping and $i^\Tt_{0\zeta}(\theta)=\theta$
and $[0,\zeta]^\Uu$ is non-dropping and $i^\Uu_{0\zeta}(\theta)=\theta$.
Now both $\Tt\rest\zeta$ and $\Uu\rest\zeta$
must be non-trivial, since $\her_\gamma^{M_1}=\her_\gamma^{N}$.
Let $\alpha+1=\mathrm{succ}^\Uu(0,\zeta]$.
Then $E^\Uu_\alpha$ is $N$-total, hence $M_1$-total.
But then by \cite[Theorem 3.5]{extmax}, $E^\Uu_\alpha\in\es_+(M^\Tt_\alpha)$, so $E^\Uu_\alpha$ is \emph{not} used in $\Uu$, a contradiction.
\end{rem}

We now move to the proof of the theorem, on mice with 2 Woodins.

\begin{proof}[Proof of Theorem \ref{tm:y=P(x)_Sigma_1^Card}]
We will use an overall similar kind of  argument to that used in the combination of \cite{power_set_Sigma_1_Card} and Remark \ref{rem:M_1} for $M_1$. But
\cite{power_set_Sigma_1_Card} focuses on
explicit ``core model'' arguments,
whereas we will handle the analogous aspects of the arugment here using methods as in Remark \ref{rem:M_1} and the proof of \cite[Theorem 5.1]{low_lev_def_above_LC}, which avoid explicitly dealing with the ``core model''.  (But in any case,
the underlying reasoning is similar.)

The basic method we will use  for identifying the power function in a $\Sigma_1^{\mathrm{Card}}$ fashion
inside a (to be chosen) premouse $M$  will be to in fact identify the extender sequence $\es^{M}$ of $M$ in such a fashion,
or equivalently, to identify the sequence of initial segments $N\ins M$ of $M$.
(This is similar to that used  for $M_1$.)
This will be enough, since inside any mouse $P$,  given any cardinal $\kappa$,
$\pow(\kappa)\cap P\sub P|\kappa^{+P}$, the initial segment of $P$ of ordinal height $\kappa^{+P}$.
Also because of this, it is enough to identify the segments $N\ins M$ such that $\OR^N$ is a cardinal of $M$. We will identify these correct initial segments $N$, as the premice $N$ of cardinal height
such that $\mathrm{Card}^N=\mathrm{Card}\cap N$
and $N$ is sufficiently iterable.
The main new obstacle in implementing this plan for $M_2$ (or some 2-Woodin premouse $M$), instead of $M_1$,
 is in giving a $\Sigma_1^{\mathrm{Card}}$ formulation of a strong enough iterability condition for (enough) proper segments of $M$, especially those below its least Woodin cardinal $\delta_0^{M}$ (this is working inside $M$).

Let us consider the formulation of iterability
for a $2$-small $\om$-sound premouse $P$ such that $\rho_\om^P=\om$ (in general we will need to deal with premice higher up, above measurable cardinals, for example, but for the purposes of the present sketch, we restict to $\rho_\om^P=\om$). We would like to identify in a simple manner an appropriate tree $T$ searching for a proof that $P$ is \emph{not} iterable,
and then our certificate of iterability would be a transitive model which can rank $T$, proving that $T$ is wellfounded.
Roughly, $T$ should build an $\om$-maximal iteration tree $\Tt$ on $P$,
embedding the tree order of $\Tt$ into the ordinals, and should build witnesses to the iterability of the Q-structures
used to guide $\Tt$ at limit stages,
and should build a proof that $\Tt$
reaches a stage from which it cannot be continued in such a manner,
or exhibit that the last model of $\Tt$ is illfounded, if it has one. The main issue here is the iterability of the Q-structures.
Toward this purpose, our hope/plan is to use the fact that $\Pi^1_2$-iterability
is a strong enough iterability condition
in the present context, and to guarantee $\Pi^1_2$-iterability via the Martin-Solovay tree $T_2$ for $\Pi^1_2$. Well, \emph{if} $\{T_2\}$ is $\Sigma_1^{\mathrm{Card}}$,
then we can use it to search for the Q-structures, expecting that $\Pi^1_2$-iterability identifies them. One further detail, though,
is that if $P$ is not iterable,
it seems it might in general be that there is some limit length
tree $\Tt$ on $P$, such that $M_1(M(\Tt))$ has no segment which is a Q-structure for $M(\Tt)$, and so satisfies ``$\delta(\Tt)$ is Woodin''.
But equivalently, there is some
$\Pi^1_2$-iterable, \emph{non}-2-small premouse $N$
such that $M(\Tt)\pins N$, $N\sats$ ``$\delta(\Tt)$ is Woodin''.
Moreover, this occurring is proof in itself that $P$ is not iterable
(since $P$ is $2$-small and $\rho_\om^P=\om$, and assuming
that the earlier Q-structures were all iterable). So our tree $T$
can simply produce this situation as a witness to non-iterability.

So a key question is whether $\{T_2\}$ is in fact $\Sigma_1^{\mathrm{Card}}$.
The author does not know whether
one can simply prove that this holds in $M_2$. But we can obtain it in some $M_2$-like inner model, and that is where we will get the power set function being $\Sigma_1^{\mathrm{Card}}$.\footnote{It seems that instead of working with $T_2$ and $\Pi^1_2$-iterability, one might try to embed the Q-structures into something which is built by $L[\es]$-construction by some model of the form $M_1^\#(X)$
for some large enough (but rather arbitrary) transitive set $X$.
But if we do not know that $\RR\sub X$,
then it is not so clear that one should be able to find such an $L[\es]$-construction in $M_1^\#(X)$,
and also, there is the challenge of first identifying the extender sequence of $M_1^\#(X)$.}

Let us fix a finite sub-theory $\varphi_0$ of the first order theory of $M_2|\delta_1^{+M_2}$, including ``I am the universe of a $2$-small premouse with 2 Woodin cardinals'', plus maybe some more statements which we will add as needed later. Say that a proper class premouse $N$ is \emph{$M_2$-like} iff $N|\delta_1^{+N}\sats\varphi_0$.

Let $(\om_1^0,u_2^0,\delta_0^0,\delta_0^{+0},\delta_1^0,\delta_1^{+0})$ be the lexicographically least tuple $(w,u,\delta,\delta',\vareps,\vareps')$
such that there is an $M_2$-like
(hence proper class) premouse $N$
such that \[ (w,u,\delta,\delta',\vareps,\vareps')=(\om_1^N,u_2^N,\delta_0^N,\delta_0^{+N},\delta_1^N,\delta_1^{+N}).\]
(This quantification over proper classes is equivalent to one over sets,
since $M_2$-like premice have only a set-sized extender sequence.)

Let $N$ be $M_2$-like,
witnessing the choice of
$(\om_1^0,u_2^0,\delta_0^0,\delta_0^{+0},\delta_1^0,\delta_1^{+0})$. We will prove that $N\sats$ ``The power set function is $\Sigma_1^{\mathrm{Card}}$'' (given that $\varphi_0$ already includes the finitely many clauses we might retroactively add later), which suffices to prove the theorem.

Work from now on in $N$.
(Of course, $N\sats$ ``I am not fully iterable'', but by our (later) choice of $\varphi_0$, $N$ satisfies various first-order consequences of iterability.)

For a real $x$ and an ordinal $\alpha$,
let $\iota_0(x,{\alpha})$ denote the least $x$-indiscernible $\mu$ such that $\mu>\alpha$.

\begin{clm}\label{clm:identify_u_2}
Let $M$ be a set-sized $2$-small premouse
such that $M\sats$ ``there are 2 Woodins'',
 $\delta_1^{+M}\in\mathrm{Card}$
 (where $\delta_1^M$ is the second Woodin of $M$)
and $M|\delta_1^{+M}\sats\varphi_0$.
Then $u_2^M=u_2$. Therefore $\{u_2\}$ is $\Sigma_1^{\mathrm{Card}}$.
\end{clm}
\begin{proof}
Since $M$ is $2$-small, $M=L_\alpha[M|\delta_1^M]$ for some $\alpha$.
Moreover, since
 $\delta_1^{+M}\in\Card$,
 note that $L[M]\sats$ ``$\delta_0^M,\delta_1^M$ are both Woodin'',
 and in fact, $M|\delta_1^{+M}$ is a cardinal segment of $L[M]$,
 and has universe $(\her_{\delta_1^{+M}})^{L[M]}$.
 So since $M|\delta_1^{+M}\sats\varphi_0$,
 $L[M]$ is $M_2$-like.

Now $\om_1^M=\om_1$. For certainly $\om_1^M\leq\om_1$,
since $M\sub V$. But then if $\om_1^M<\om_1$, the properties of $M$ contradict the minimality of $\om_1^0$.

Now let us see that $u_2^M=u_2$.
Certainly $u_2^M\leq u_2$,
since $M$ is correct about sharps for reals, and $M\sats$ ``$u_2^M=\sup_{x\in\RR}\iota_0(x,\om_1)$''.
But  like for $\om_1^M$ and since $\om_1^M=\om_1^0$,
if we had $u_2^M<u_2$, then the properties of $M$ would contradict the lexicographic minimality of $(\om_1^0,u_2^0)$.
\end{proof}

\begin{clm}\label{clm:T_2_is_Sigma_1^Card}
For any  $M$ as in Claim \ref{clm:identify_u_2}, we have $T_2^M=T_2$. Therefore $\{T_2\}$ is $\Sigma_1^{\mathrm{Card}}$.
\end{clm}
\begin{proof}
We already know $u_2^M=u_2$
and $u_1^M=\om_1^M=\om_1=u_1$.
Standard calculations with indiscernibles
give then that $u_n^M=u_n$ for all $n\in[1,\om)$. (That is,
for each $n$, $u_{n+1}=\sup_{x\in\RR}\iota_0(x,u_n)$. And $M$ satisfies that the same characterization holds. But
\[ \{\iota_0(x,\om_1)\bigm|x\in\RR\cap M\}\text{ is cofinal in }\{\iota_0(x,\om_1)\bigm|x\in\RR\},\]
since $u_2^M=u_2$ is the supremum of both of these sets. And for all reals $x,y$,
\[ \iota_0(x,\om_1)<\iota_0(y,\om_1)\iff\iota_0(x,u_n)<\iota_0(y,u_n),\]
by indiscernibility. It follows that
\[ \{\iota_0(x,u_n)\bigm|x\in\RR\cap M\}\text{ is cofinal in }\{\iota_0(x,u_n)\bigm|x\in\RR\},\]
and so $u_{n+1}^M=u_n=$ the common supremum of these sets.)

Now since every ordinal $<u_\om$
is of the form \[ t^{L[x]}(x,u_1,\ldots,u_n)=t^{L[x]}(x,u_1^M,\ldots,u_n^M) \]
for some $x\in\RR\cap M$
and $n<\om$, it follows that $T_2^M=T_2$.
\end{proof}

Using $T_2$, we can of course
define the standard tree projecting to a universal $\Sigma^1_3$ set.

Let $P$ be a sound $2$-small premouse
such that $\rho_\om^P=\om$.
We define the tree $T_P$ as follows;
it is a tree of attempts to prove that $P$ is \emph{not} $(\om,\om_1)$-iterable.
The tree $T_P$ is the natural tree of attempts to build a pair $(x,\pi)$ such that  $x\in\RR$, $\pi:\om\to u_\om$ and $(x,\pi)$ codes
the following objects:
\begin{enumerate}[label=(\arabic*)]
\item a countable $\om$-maximal
tree $\Tt$ on $P$, together with an order-preserving embedding  $\sigma:\lh(\Tt)\to\om_1$,
\item for each limit $\eta<\lh(\Tt)$,
letting $Q=Q(\Tt\rest\eta,[0,\eta)^\Tt)$,
 a witness $w_Q$ that $Q$ is above-$\eta$, $k$-$\Pi^1_2$-iterable,
where $k$ is least such that $\rho_{k+1}^Q\leq\delta(\Tt\rest\eta)$,
and $w_Q$ is the appropriate kind of branch through $T_2$,
\item if $\Tt$ has successor length,
an infinite descending sequence through $\OR^{M^\Tt_\infty}$,
\item if $\Tt$ has limit length,
a premouse $Q_\infty$ such that $M(\Tt)\ins Q_\infty$
and $Q_\infty\sats$ ``$\delta(\Tt)$ is Woodin'', and a witness $w^{Q_\infty}$
showing that $Q_\infty$ is above-$\delta(\Tt)$, $k$-$\Pi^1_2$-iterable,
where either:
\begin{enumerate}
\item $Q_\infty$ is non-$1$-small above $\delta(\Tt)$ and $k=0$, or
\item $Q_\infty$ is $1$-small above $\delta(\Tt)$ and $Q_\infty$ is $\delta(\Tt)$-sound and $k$ is least such that $\rho_{k+1}^{Q_\infty}\leq\delta(\Tt)$
\end{enumerate}
\item if $\Tt$ has limit length,
a structure $R$ for LST
such that
 $\Tt,Q_\infty\in\HC^R$ and $R\sats\ZF^-$
 and $R\sats$ ``there is no $\Tt$-cofinal branch $b$ such that $Q_\infty\ins M^\Tt_b$'',
 together with an order-preserving embedding $\pi:\OR^R\to\om_1$.
\end{enumerate}

\begin{clm}\label{clm:om-pm_P_it_iff_T_P_wfd}
$P$ is $(\om,\om_1)$-iterable
iff $T_P$ is wellfounded.
\end{clm}
\begin{proof}
Suppose $P$ is $(\om,\om_1)$-iterable,
but $x$ is a branch through $T_P$,
giving objects $\Tt,\sigma$, and if $\Tt$ has limit length, $Q_\infty,R,\pi$.
Note first that $\Tt$ is formed according to the unique strategy for $P$. For given a limit $\eta<\lh(\Tt)$,
we know that $Q=Q(\Tt\rest\eta,[0,\eta)^\Tt)$ is $k$-$\Pi^1_2$-iterable above $\delta(\Tt\rest\eta)$, for the relevant $k$. But because $P$ is $2$-small and fully iterable,
this implies that $Q$ is the correct (fully iterable) Q-structure,
and $[0,\eta)^\Tt$ is the correct branch. So if $\Tt$ had successor length, then $M^\Tt_\infty$ would be wellfounded, a contradiction.
So it has limit length. But then again,
$Q_\infty$ is $1$-small and is
an initial segment of the correct Q-structure,
so there \emph{is} a $\Tt$-cofinal branch with $Q_\infty\ins M^\Tt_b$, a contradiction. (It can't be that $Q_\infty$ is non-$1$-small,
since if it is, then $M_1(M(\Tt))|\delta(\Tt)^{+M_1(M(\Tt))}\ins Q_\infty$,
and so $M_1(M(\Tt))\sats$ ``$\delta(\Tt)$ is Woodin'', but this contradicts the $2$-smallness of $P$.)

Conversely, suppose $P$ is not $(\om,\om_1)$-iterable. Then the putative strategy for $P$, guided by Q-structures
which are iterable-above-$\delta(\Tt)$,
cannot succeed. Taking a (countable length)
tree witnessing this of minimal length,
we can use it to build a branch through $T_P$, as desired.
\end{proof}

We may assume (for $P$ still as above):
\begin{clm}\label{clm:om_1-it_iff_om_1+1-it} $P$ is $(\om,\om_1+1)$-iterable
iff $P$ is $(\om,\om_1)$-iterable.
\end{clm}
\begin{proof}
The corresponding statement holds in
 $M_2$, so we can build in into ``$M_2$-like''.
\end{proof}
\begin{clm}
$\{N|\om_1^N\}$ is $\Sigma_1^{\Card}$.
\end{clm}
\begin{proof}
By Claims \ref{clm:om-pm_P_it_iff_T_P_wfd} and \ref{clm:om_1-it_iff_om_1+1-it},
$N|\om_1^N$ is the unique 2-small premouse
$M$ which satisfies $\ZFC^-$ + ``Every set is countable'',
with $\OR^M=\om_1$,
and such that there is a transitive set $R\sats\ZFC^-$
with $T_2,M\in R$
and for every $P\pins M$ with $\rho_\om^P=\om$, we have $T_P\in R$
and $R\sats$ ``$T_P$ is wellfounded''.
By Claim \ref{clm:T_2_is_Sigma_1^Card}, this gives a $\Sigma_1^{\Card}$ definition of $\{N|\om_1^N\}$.
\end{proof}

\begin{clm}\label{clm:it_above_theta_implies_seg}
Let $\theta<\delta_1^N$ be a cardinal.
Let $P$ be a sound premouse such that $N|\theta\pins P$ and $\rho_\om^P=\theta$. Suppose that $P$ is above-$\theta$, $(\om,\theta^++1)$-iterable. Then $P\pins N$.
\end{clm}
\begin{proof}
Suppose not and consider the comparison of $N$ versus $P$, producing trees $\Tt$ on $N$ and $\Uu$ on $P$, stopping at stage $\theta^++1$, or at the first stage $\alpha$ such that $E^\Uu_\alpha\neq\emptyset$ and $\crit(E^\Uu_\alpha)<\theta$. (Our hypothesis on the iterability of $P$ ensures that this makes sense,
since $N$ is sufficently self-iterable.)

Suppose that the comparison does not reach a successful termination.
It can't be that $\alpha<\theta^+$
(with $\crit(E^\Uu_\alpha)<\theta$).
For if this holds, then since we are using the correct strategy for $N$,
and $\lh(E^\Uu_\alpha)$ is the index of the least disagreement between $M^\Tt_\alpha$ and $M^\Uu_\alpha$,
\cite[Theorem 3.5]{extmax}
shows that $E^\Uu_\alpha\in\es_+(M^\Tt_\alpha)$, contradicting that it indexes the least disagreement.

So $\alpha=\theta^+$. In this case we use a typical universal weasel argument.
That is, since $\card(P)<\theta^+$,  $\Uu$ uses $\theta^+$-many extenders, with $\delta(\Uu)=\theta^+$,
so $\theta^+$ is a limit cardinal of $M^\Uu_{\theta^+}$ and there is $\xi<^\Uu\theta^+$
and $\gamma<\theta^+$
such that $(\xi,\theta^+)^\Uu$ does not drop and
$i^\Uu_{\xi\theta^+}(\gamma)=\theta^+$.
 We claim that ($*$) $\Tt$ also uses $\theta^+$-many extenders,
and there is $\xi'<^\Tt\theta^+$
and $\gamma'<\theta^+$
such that  $i^\Tt_{\xi'\theta^+}(\gamma')=\theta^+$.
For $\theta$ is the largest cardinal of $N|\theta^+$, and so ($*$) is easy to see unless $\kappa=\crit(E^\Tt_\alpha)<\theta$
where  $\alpha+1=\mathrm{succ}^\Tt(0,\theta^+)$,
so suppose this is the case.
If $\cof(\theta)\neq\kappa$
then $i^\Tt_{0,\alpha+1}(\theta^+)=\theta^+$,
and since then then it is again easy to see that ($*$) holds.
If instead $\cof(\theta)=\kappa$
then we get $i^\Tt_{0,\alpha+1}(\theta)>\theta^+$,
but as in \cite{Kwithoutmeas} (see also \cite[\S3.5]{downward_transference_of_mice}),
$M^{\Tt}_{\alpha+1}|\theta^+\sats$ ``$\theta$ is the largest cardinal and $\cof(\theta)=\kappa$'',
from which ($*$) easily follows.
We now get the usual kind of contradiction to termination of comparison of weasels.

So the comparison must terminate successfully
at some stage $\alpha<\theta^+$.
Suppose $\Uu$ is non-trivial.
Then (since $\Uu$ is above $\theta$)
$b^\Uu$ drops, $b^\Tt$
does not, and $M^\Tt_\infty\ins M^\Uu_\infty$. But this contradicts
the fact that $\OR^P<\theta^+<\OR^N$.
So $\Uu$ is trivial, and note then that $P\pins N$, as desired.
\end{proof}

Fix a sound 2-small premouse $P$
and let $\theta=\rho_\om^P$.
We define a tree $T_P$;
much as before, it is a tree of attempts to prove that $P$ is \emph{not}
above-$\theta$ iterable.
Define this $T_P$ just as before,
except that instead of building $\Tt$ on $P$,
$T_P$ also searches for a countable structure $\bar{P}$ and an elementary $\varrho:\bar{P}\to P$, with $\theta\in\rg(\varrho)$, and builds $\Tt$ on $\bar{P}$, above $\varrho^{-1}(\theta)$.
(So $T_P$ is a tree on $\om\cross\max(u_\om,\OR^P)$.)
\begin{clm}\label{clm:P_below_delta_0^N_it_iff_T_P_wfd}
Let $\theta,P,T_P$ be as in the previous paragraph. Then:
\begin{enumerate}
\item If $P$ is above-$\theta$, $(\om,\theta^++1)$-iterable
then  $T_P$ is wellfounded.
\item If $\theta<\delta_0^N$ and $T_P$ is wellfounded
then $P$ is above-$\theta$, $(\om,\theta^++1)$-iterable.
\end{enumerate}
\end{clm}
\begin{proof}
If $P$ is iterable above $\theta$,
then any $\bar{P}$ (as built by $T_P$) is also iterable above $\bar{\theta}=\varrho^{-1}(\theta)$,
which implies that $T_P$ is wellfounded, much as before. Conversely, suppose that $\theta<\delta_0^N$, but $P$ is not above-$\theta$, $(\om,\theta^++1)$-iterable above $\theta$.
Then we can fix an iteration tree witnessing this, together with the canonical Q-structures (or $M_1^\#(M(\Tt))$ satisfying ``$\delta(\Tt)$ is Woodin''), since $N|\delta_0^N$ is sufficiently closed to compute these (as $\theta<\delta_0^N$). But then we can take a countable hull of these objects to produce a branch through $T_P$.
\end{proof}

Now fix $\theta,P$ again as above,
and suppose that $P\sats$ ``there is a Woodin cardinal $\leq\theta$''.
Then we define a tree $T_P'$,
similar to $T_P$,
but now the Q-structures involved
are  just of form $\mathcal{J}_\alpha(M(\Tt))$ for some ordinal $\alpha$,
and the special case of producing
$M_1^\#(M(\Tt))$ satisfying ``$\delta(\Tt)$ is Woodin''
is replaced by $M(\Tt)^\#\sats$ ``$\delta(\Tt)$ is Woodin''.

\begin{clm}
Let $\theta,P,T'_P$ be as in the previous paragraph. Then:
\begin{enumerate}
\item If $P$ is above-$\theta$, $(\om,\theta^++1)$-iterable
then  $T'_P$ is wellfounded.
\item If $\theta<\delta_1^N$ and $T'_P$ is wellfounded
then $P$ is above-$\theta$, $(\om,\theta^++1)$-iterable.
\end{enumerate}
\end{clm}
\begin{proof}
This is a simplification of the proof of Claim \ref{clm:P_below_delta_0^N_it_iff_T_P_wfd}.
\end{proof}

\begin{clm}\label{clm:any_good_2-small_premouse_is_N} Let $M$ be a $2$-small premouse satisfying ``there are 2 Woodin cardinals and $\delta_1^+$ exists'',
$M|\delta_1^{+M}\sats\varphi_0$,
$\delta_1^{+M}\in\Card$,
and $\Card\cap\delta_1^{+M}=\Card^M\cap\delta_1^{+M}$,
and for every cardinal $\theta<\delta_1^M$ and $P$
such that $P\ins M$ and $\rho_\om^P=\theta$, we have:
\begin{enumerate}
\item if $\theta<\delta_0^M$
then $T_P$ is wellfounded, and
\item if $\delta_0^M\leq \theta<\delta_1^M$ then $T'_P$ is wellfounded.
\end{enumerate}
Then $M|\delta_1^M=N|\delta_1^N$.
\end{clm}
\begin{proof}
As in the proof of Claim \ref{clm:identify_u_2},
$M$ is $M_2$-like, $\om_1^M=\om_1$ and $u_2^M=u_2$.
So by minimality, $\delta_0^N\leq\delta_0^M$,
and if $\delta_0^N=\delta_0^M$
(and hence $\delta_0^{+N}=\delta_0^{+M}$, by the correctness of $M$'s cardinals),
then $\delta_1^N\leq\delta_1^M$,
and if $\delta_1^N=\delta_1^M$
then $\delta_1^{+N}=\delta_1^{+M}$.

Now $M|\delta_0^M=N|\delta_0^N$.
For let $\theta<\delta_0^N$
be a cardinal and suppose $M|\theta=N|\theta$, and let $P\pins M$ with $\rho_\om^P=\theta$.
Then by hypothesis
(and since $\delta_0^N\leq\delta_0^M$),
$T_P$ is wellfounded, so by Claims
\ref{clm:it_above_theta_implies_seg} and \ref{clm:P_below_delta_0^N_it_iff_T_P_wfd},
$P\pins N$, as desired. Since also $\theta^{+M}=\theta^{+N}$, it follows
that $M|\theta^{+M}=N|\theta^{+N}$.

Now since $M|\delta_0^N=N|\delta_0^N$
and $M\sub N$, it follows that $\delta_0^M$ is Woodin in $M$.
So $\delta_0^M=\delta_0^N$,
and so $\delta_0^{+M}=\delta_0^{+N}$.

A similar calculation as before
now gives that $M|\delta_1^N=N|\delta_1^N$,
and then again since $M\sub N$,
it follows that $\delta_1^M=\delta_1^N$
and $\delta_1^{+M}=\delta_1^{+N}$,
as desired.
\end{proof}

\begin{clm}
$\{N|\delta_1^N\}$ is $\Sigma_1^{\Card}$.
\end{clm}
\begin{proof}
This follows easily from the preceding claims. Our $\Sigma_1^{\Card}$ formula  $\varphi(M)$ (in variable $M$)
just demands that $M$ satisfies the
 hypothesis of Claim \ref{clm:any_good_2-small_premouse_is_N},
 by asserting that there is a transitive model which ranks the relevant trees $T_P$ and $T'_P$;
by earlier claims, we can identify these trees in a  $\Sigma_1^{\Card}$ fashion. We then get that $\varphi(M)$ holds iff $M=N|\delta_1^N$.
\end{proof}

It now follows that $N\sats$ ``The power set function is $\Sigma_1^{\Card}$'',
 like in the proof for $L$:
 Given $x,y$, we have $y=\pow(x)$
 iff there is a cardinal $\theta$
 such that $x,y\in L_\theta(N|\delta_1^N)$
 and $y=\pow(x)\cap L_\theta(N|\delta_1^N)$. This completes the proof.
\end{proof}

\begin{ques}
Suppose $M_2^\#$ exists and is $(0,\om_1+1)$-iterable. Does $M_2$ satisfy ``power is $\Sigma_1^{\Card}$?
\end{ques}

\begin{ques}
Is the theory ZFC + ``there are 2 Woodin cardinals and a measurable above'' + ``power is $\Sigma_1^{\Card}$'' consistent (relative to large cardinals)?
\end{ques}

\section*{Acknowledgements}

This research was funded by the Austrian Science Fund (FWF) [10.55776/Y1498].

\bibliographystyle{plain}
\bibliography{../bibliography/bibliography}

\end{document}